\documentclass[11pt,a4paper]{amsart}
\usepackage{graphicx,multirow,array,amsmath,amssymb}

\newtheorem{theorem}{Theorem}

\newtheorem{lemma}[theorem]{Lemma}

\begin{document}

\title{Quantum knot mosaics and the growth constant}

\author[S. Oh]{Seungsang Oh}
\address{Department of Mathematics, Korea University, Seoul 02841, Korea}
\email{seungsang@korea.ac.kr}

\thanks{This work was supported by the National Research Foundation of Korea(NRF) grant funded
by the Korea government(MSIP) (No. NRF-2014R1A2A1A11050999).}

\begin{abstract}
Lomonaco and Kauffman introduced a knot mosaic system to give a precise  and workable definition 
of a quantum knot system,
the states of which are called quantum knots.
This paper is inspired by an open question about the knot mosaic enumeration suggested by them.
A knot $n$--mosaic is an $n \times n$ array of 11 mosaic tiles 
representing a knot or a link diagram by adjoining properly that is called suitably connected.
The total number of knot $n$--mosaics is denoted by $D_n$
which is known to grow in a quadratic exponential rate.
In this paper, we show the existence of the knot mosaic constant 
$\delta = \lim_{n \rightarrow \infty} D_n^{\ \frac{1}{n^2}}$
and prove that 
$$4 \leq \delta \leq \frac{5+ \sqrt{13}}{2} \ (\approx 4.303).$$
\end{abstract}

\maketitle

\section{Preliminaries}

The quantum knot system was developed by Lomonaco and Kauffman 
to explain how to make quantum information versions
of mathematical structures in~\cite{LK1, LK2}.
They build a knot mosaic system to set the foundation for a quantum knot system,
based on the planar projections of knots and the Reidemeister moves.

Throughout this paper the term `knot' means either a knot or a link.
An example of a knot mosaic is shown in Figure~\ref{fig1}~(a).
Knot mosaics are constructed by using 11 mosaic tiles, listed in Figure~\ref{fig1}~(b).

\begin{figure}[h]
\includegraphics{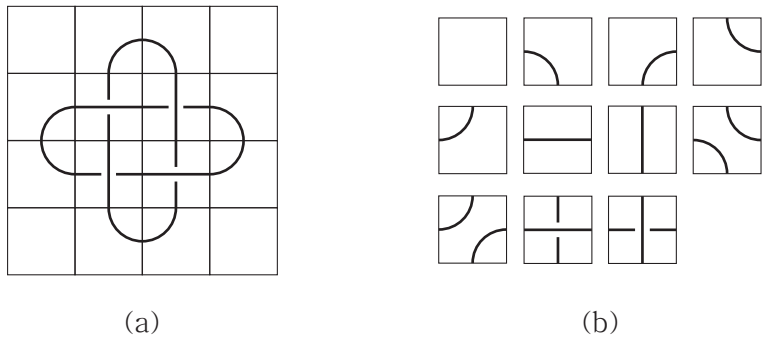}
\caption{An example of a knot mosaic and 11 mosaic tiles}
\label{fig1}
\end{figure}

This paper is inspired by an open question (9) about the knot mosaic enumeration proposed in~\cite{LK2}.
The enumeration of knot mosaic is not only an interesting problem in its own right 
but is also of considerable importance in the quantum knot theory.
Let $D_n$ denote the total number of knot $n$--mosaics.
The author, Hong, Lee and Lee announced several results on $D_n$ in the series of 
papers~\cite{HLLO1, HLLO2, LHLO, OHLL}.
Based upon the results, $D_n$ is known to grow in a quadratic exponential rate.
We consider the behavior of the growth rate.
The limit, if it exists, 
$$\delta = \lim_{n \rightarrow \infty} D_n^{\ \frac{1}{n^2}}$$
is called the {\em knot mosaic constant\/}.

\begin{theorem}\label{thm:growth}
The knot mosaic constant $\delta$ exists.
Furthermore, 
$$4 \leq \delta \leq \frac{5+ \sqrt{13}}{2} \ (\approx 4.303).$$
\end{theorem}

As a previous result,
lower and upper bounds on $D_n$ for $n \geq 3$ were established as follows in~\cite{HLLO1}:
\begin{equation}
\frac{2}{275}(9 \cdot 6^{n-2} + 1)^2 \cdot 2^{(n-3)^2} \ \leq \ D_n \
\leq \ \frac{2}{275}(9 \cdot 6^{n-2} + 1)^2 \cdot (4.4)^{(n-3)^2} \tag{$*$}
\end{equation}
These bounds suggested that $\delta$ lies between $2$ and $4.4$.

This paper is organized as follows.
In Section 2, we give precise definition of knot mosaics with a slight generalization
and previously known results about the enumeration of knot mosaics.
In Section 3, the existence of the knot mosaic constant $\delta$ is provided by applying
an extended version of Fekete's Lemma.
In Section 4, we rigorously find lower and upper bounds of $\delta$ with heavy reliance on
the main theorem of~\cite{OHLL}.

\section{Enumeration of knot mosaics}

We begin by presenting the basic notion of knot mosaics and
then introduce previously known results about the enumeration of knot mosaics. \\

\noindent {\bf Definition 1 }
For positive integers $m$ and $n$,
an {\em $(m,n)$--mosaic\/} is an $m \times n$ array $M=(M_{ij})$ of 11 mosaic tiles depicted in Figure~\ref{fig1}~(b). \\

This definition is a rectangular version of the definition of an {\em $n$--mosaic\/}
that is an $n \times n$ array of mosaic tiles.
Obviously the set of all $(m,n)$--mosaics has $11^{mn}$ elements.

A {\em connection point\/} of a mosaic tile is defined as the midpoint of a mosaic tile edge
that is also the endpoint of a curve drawn on the tile.
Then the first mosaic tile has zero, the next six tiles with exactly one curve inside have two,
and the last four tiles have four connect points.
A mosaic is called {\em suitably connected\/} if any pair of mosaic tiles
lying immediately next to each other in either the same row or the same column
have or do not have connection points simultaneously on their common edge. \\

\noindent {\bf Definition 2 }
A {\em knot $(m,n)$--mosaic\/} is a suitably connected $(m,n)$--mosaic
which has no connection point on the boundary edges.
$D_{m,n}$ denotes the total number of knot $(m,n)$--mosaics.
Note that $D_n = D_{n,n}$. \\

A knot $(m,n)$--mosaic represents a specific knot or link diagram.
A knot $(n,n)$--mosaic is simply specified by a {\em knot $n$--mosaic\/}.
Two examples of mosaics are provided in Figure~\ref{fig2}.

\begin{figure}[h]
\includegraphics{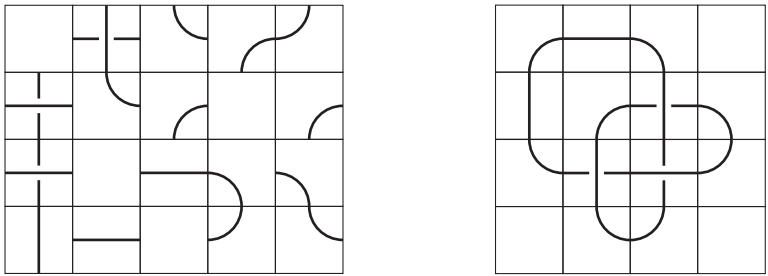}
\caption{Examples of a non-knot $(4,5)$--mosaic and the trefoil knot 4--mosaic}
\label{fig2}
\end{figure}

The problem of calculating $D_{m,n}$ is one of simplicity of definition but hiding difficulty of solution,
due to its non-Markovian processing.
$D_{1,n} = 1$ and $D_{2,n} = 2^{n-1}$ for a positive integer $n$, and $D_{3,3} = 22$.
The problem becomes increasingly difficult thereafter.
Refer~\cite{HLLO2} for a table of the precise values of $D_{m,n}$ for $m,n = 4,5,6$.

In~\cite{OHLL}, an algorithm producing the precise value of $D_{m,n}$ was proposed as follows:

\begin{theorem}[Oh--Hong--Lee--Lee] \label{thm:state}
For integers $m,n \geq 3$,
$$ D_{m,n} = 2 \, \| (X_{m-2}+O_{m-2})^{n-2} \| $$
where $X_{m-2}$ and $O_{m-2}$ are $2^{m-2} \times 2^{m-2}$ matrices recursively defined by
$$ X_{k+1} = 
\begin{bmatrix} X_k & O_k \\ O_k & X_k  \end{bmatrix}
\ \mbox{and } \
O_{k+1} = 
\begin{bmatrix} O_k & X_k \\ X_k & 4 \, O_k  \end{bmatrix} $$
for $k=1, \cdots, m \! - \! 3$, starting with
$X_1 = \begin{bmatrix} 1 & 1 \\ 1 & 1 \end{bmatrix}$ and \
$O_1 = \begin{bmatrix} 1 & 1 \\ 1 & 4 \end{bmatrix}$.
\end{theorem}

Here $\| N \|$ denotes the sum of all entries of a matrix $N$.
Due to Theorem~\ref{thm:state}, we get Table~\ref{tab1} of the precise values of $D_n$ 
and approximated values of $D_n^{\ \frac{1}{n^2}}$.
This observation that $D_n^{\ \frac{1}{n^2}}$ steadily increases is of considerable significance. \\

\begin{table}[h]
\begin{tabular}{crc}      \hline \hline
\ \ $n$ \ \ & $D_n$ & \ \ $D_n^{\ \frac{1}{n^2}}$ \ \ \\    \hline
1 & 1 & 1.000 \\
2 & 2 & 1.189 \\
3 & 22 & 1.410 \\
4 & 2594 & 1.634 \\
5 & 4183954 & 1.840 \\
6 & 101393411126 & 2.022 \\
7 & 38572794946976686 & 2.180 \\
8 & 234855052870954505606714 & 2.318 \\
9 & 23054099362200397056093750003442 & 2.439 \\  \hline \hline
\end{tabular}
\vspace{2mm}
\caption{List of $D_n$ and approximated values of $D_n^{\ \frac{1}{n^2}}$}
\label{tab1}
\end{table}

\section{Existence of the knot mosaic constant}

In this section, we prove the existence of the knot mosaic constant $\delta$
that is the first part of Theorem~\ref{thm:growth}.
It is due to the two-variable version of Fekete's Lemma
introduced in~\cite{OhD1}.

\begin{lemma}[Two-variable Fekete's Lemma] \label{lem:Fekete}
A positive valued double sequence $\{ a(m,n) \}_{m, \, n \in \, \mathbb{N}}$ is
superadditive in both indices as
$a(m_1 \! + \! m_2,n) \geq a(m_1,n) + a(m_2,n)$ and $a(m,n_1 \! + \! n_2) \geq a(m,n_1) + a(m,n_2)$.
Then
$$ \lim_{n \rightarrow \infty} \frac{a(n,n)}{n^2} \ = \ \sup_{n \in \, \mathbb{N}} \frac{a(n,n)}{n^2}$$
if the supremum exists.
\end{lemma}

\begin{proof}
We merely follows the proof of Fekete's Lemma.
Let $S = \sup_n \frac{a(n,n)}{n^2}$ and let $B$ be any number less than $S$.
Choose $k \geq 1$ such that $ B < \frac{a(k,k)}{k^2}$.
For $n > k$, there are integers $p_n$ and $q_n$ such that $n = p_n k + q_n$ and $0 \leq q_n< k$
by the division algorithm.
Applying the definition of superadditivity many times in both indices, we obtain
\begin{equation*} \begin{aligned}
a(n,n) & \geq \ p_n a(k, p_n k \! + \! q_n) + a(q_n, p_n k \! + \! q_n)  \\
          & \geq p_n^2 a(k,k) + p_n a(k,q_n) + p_n a(q_n, k) + a(q_n, q_n).
\end{aligned} \end{equation*}
So,
$$ \frac{a(n,n)}{n^2} \ \geq \ \Big(\frac{p_n k}{n}\Big)^2 \, \frac{a(k,k)}{k^2}.$$
Since $\frac{p_n k}{n} \rightarrow 1$ as $n \rightarrow \infty$,
we have
$$ B < \frac{a(k,k)}{k^2} \leq \lim_{n \rightarrow \infty} \frac{a(n,n)}{n^2} \leq S.$$
Finally, let $B$ go to $S$ and we obtain
$$ \lim_{n \rightarrow \infty} \frac{a(n,n)}{n^2} = S = \sup_{n \in \, \mathbb{N}} \frac{a(n,n)}{n^2}.$$
\end{proof}

To show the existence of the limit of $D_n^{\ \frac{1}{n^2}} = (D_{n,n})^{\frac{1}{n^2}}$,
we only need to verify that $D_{m,n}$ satisfies the supermultiplicative property in both indices as
$D_{m_1+m_2,n} \geq D_{m_1,n} \cdot D_{m_2,n}$ and $D_{m,n_1+n_2} \geq D_{m,n_1} \cdot D_{m,n_2}$.
This asserts that $\log D_{m,n}$ is a superadditive function in both indices.
Furthermore the inequality $(*)$ guarantees the existence of $\sup_n \frac{\log D_{n,n}}{n^2} \leq \log 4.4$.
Then we can apply Lemma~\ref{lem:Fekete}.

The supermultiplicative inequalities for $D_{m,n}$ are obvious 
because we can get a new knot $(m,n_1 \! + \! n_2)$--mosaic
by simply adjoining two knot $(m,n_1)$-- and $(m,n_2)$--mosaics
as drawn in Figure~\ref{fig3}.

\begin{figure}[h]
\includegraphics{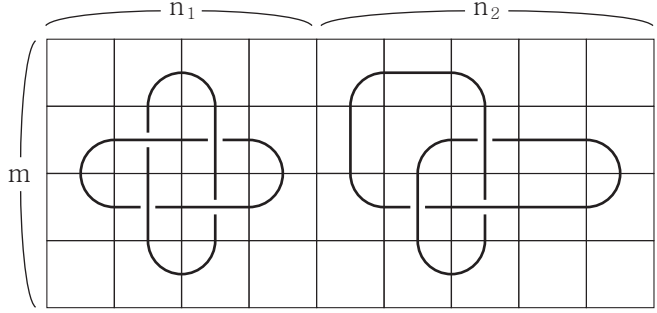}
\caption{$D_{m,n_1+n_2} \geq D_{m,n_1} \cdot D_{m,n_2}$}
\label{fig3}
\end{figure}

\section{Bounds of the knot mosaic constant}

To complete the proof of Theorem~\ref{thm:growth}, 
we find lower and upper bounds of the knot mosaic constant $\delta$.
This procedure heavily relies on Theorem~\ref{thm:state}.

Define the $2^{n-2} \times 2^{n-2}$ matrices $\underline{X}_{n-2}$ and $\underline{O}_{n-2}$ by the recurrence relations
$$ \underline{X}_{k+1} = 
\begin{bmatrix} \mathbb{O} & \mathbb{O} \\ 
\mathbb{O} & \mathbb{O}  \end{bmatrix}
\ \mbox{and } \
\underline{O}_{k+1} = 
\begin{bmatrix} \mathbb{O} & \mathbb{O} \\ 
\mathbb{O} & 4 \, \underline{O}_k  \end{bmatrix} $$
for $k=1, \cdots, n-3$, starting with
$\underline{X}_1 = \begin{bmatrix} 0 & 0 \\ 0 & 0 \end{bmatrix}$ and \
$\underline{O}_1 = \begin{bmatrix} 0 & 0 \\ 0 & 4 \end{bmatrix}$. \\
Here $\mathbb{O}$ indicates the square zero-matrix with appropriate size.
Indeed $(2,2)$--quadrant $4 \, O_k$ of $O_{k+1}$ in Theorem~\ref{thm:state}
has effect on the dominant asymptotic behavior.
Then, 
$$ \underline{D}_n  \mathrel{\mathop:}= 2 \, \| (\underline{X}_{n-2}+\underline{O}_{n-2})^{n-2} \| \ \leq \ D_n. $$
Since the matrix $\underline{X}_{n-2}+\underline{O}_{n-2}$ has all entries 0, 
except that its $(2^{n-2},2^{n-2})$--entry is exactly $4^{n-2}$,
$$ \underline{\delta} \mathrel{\mathop:}= \lim_{n \rightarrow \infty} \underline{D}_n^{\ \frac{1}{n^2}} = 4.$$

Furthermore, define the matrices  $\overline{X}_{n-2}$ and $\overline{O}_{n-2}$ by 
$$ \overline{X}_{k+1} = 
\begin{bmatrix} \overline{X}_k & \overline{O}_k \\ 
\overline{O}_k & 4 \, \overline{X}_k  \end{bmatrix}
\ \mbox{and } \
\overline{O}_{k+1} = 
\begin{bmatrix} \overline{O}_k & \overline{X}_k \\ 
\overline{X}_k & 4 \, \overline{O}_k  \end{bmatrix} $$
for $k=1, \cdots, n-3$, starting with
$\overline{X}_1 = \begin{bmatrix} 1 & 1 \\ 1 & 4 \end{bmatrix}$ and \
$\overline{O}_1 = \begin{bmatrix} 1 & 1 \\ 1 & 4 \end{bmatrix}$. 
Or, we may start at $k=0$ with $\overline{X}_0 = \begin{bmatrix} 1 \end{bmatrix}$ and \
$\overline{O}_0 = \begin{bmatrix} 1 \end{bmatrix}$. 
Then, 
$$ \overline{D}_n \mathrel{\mathop:}= 2 \, \| (\overline{X}_{n-2}+\overline{O}_{n-2})^{n-2} \| \ \geq \ D_n.$$

Considering the matrix $N_k = \overline{X}_k + \overline{O}_k$,
\begin{equation*} \begin{aligned}
\| (N_{k+1})^s\| & = \left\| \begin{bmatrix} N_k & N_k \\ N_k & 4 \, N_k \end{bmatrix}^s \right\| 
= \left\| \begin{bmatrix} a_s (N_k)^s & b_s (N_k)^s \\ c_s (N_k)^s & d_s (N_k)^s \end{bmatrix} \right\| \\
& = (a_s + b_s + c_s + d_s) \| (N_k)^s \| 
=  \left\| \begin{bmatrix} 1 & 1 \\ 1 & 4 \end{bmatrix}^s \right\| \cdot \|(N_k)^s\|
\end{aligned} \end{equation*}
where $\begin{bmatrix} a_s & b_s \\ c_s & d_s \end{bmatrix} = 
\begin{bmatrix} 1 & 1 \\ 1 & 4 \end{bmatrix}^s$ for a positive integer $s$.
Then,
\begin{equation*} \begin{aligned}
\overline{D}_{n+2} & = 2 \| (N_n)^n \| = 2 \left\| \begin{bmatrix} 1 & 1 \\ 1 & 4 \end{bmatrix}^n \right\| 
\cdot \left\| (N_{n-1})^n \right\| \\
& = 2 \left\| \begin{bmatrix} 1 & 1 \\ 1 & 4 \end{bmatrix}^n \right\|^n \cdot \| (N_0)^n \| 
= 2^{n+1} \left\| \begin{bmatrix} 1 & 1 \\ 1 & 4 \end{bmatrix}^n \right\|^n.
\end{aligned} \end{equation*}

Owing to the equality $\begin{bmatrix} 1 & 1 \\ 1 & 4 \end{bmatrix} =
P \cdot \begin{bmatrix} \lambda_1 & 0 \\ 0 & \lambda_2 \end{bmatrix} \cdot P^{-1}$
where $\lambda_1$ and $\lambda_2$ are the eigenvalues $\frac{5 - \sqrt{13}}{2}$ and $\frac{5 + \sqrt{13}}{2}$, 
respectively, of the matrix $\begin{bmatrix} 1 & 1 \\ 1 & 4 \end{bmatrix}$
and $P$ is the diagonalizing matrix,
\begin{equation*} \begin{aligned}
\overline{\delta} \mathrel{\mathop:}= & \lim_{n \rightarrow \infty} \overline{D}_n^{\ \frac{1}{n^2}} 
= \lim_{n \rightarrow \infty} \overline{D}_{n+2}^{\ \frac{1}{n^2}} 
= \lim_{n \rightarrow \infty} \left\| \begin{bmatrix} 1 & 1 \\ 1 & 4 \end{bmatrix}^n \right\|^{\ \frac{1}{n}} \\
=  & \lim_{n \rightarrow \infty} \left\| P \cdot \begin{bmatrix} 
\lambda_1^n & 0 \\ 0 & \lambda_2^n \end{bmatrix} \cdot P^{-1} \right\|^{\ \frac{1}{n}} 
= \lambda_2 = \frac{5 + \sqrt{13}}{2}.
\end{aligned} \end{equation*}

As a conclusion, the inequality $ \underline{D}_n \, \leq \, D_n \, \leq \, \overline{D}_n$
guarantees the following bounds of the knot mosaic constant as desired,
$$\underline{\delta} \, \leq \, \delta \, \leq \, \overline{\delta}.$$


\begin{thebibliography}{AA}
\bibitem{CL} M. Carlisle and M. Laufer,
    {\em On upper bounds for toroidal mosaic numbers},
    Quantum Inf. Process. \textbf{12} (2013) 2935--2945.
\bibitem{HLLO1} K. Hong, H. Lee, H. J. Lee and S. Oh,
    {\em Upper bound on the total number of knot $n$--mosaics},
    J. Knot Theory Ramifications \textbf{23} (2014) 1450065.
\bibitem{HLLO2} K. Hong, H. Lee, H. J. Lee and S. Oh,
    {\em Small knot mosaics and partition matrices},
    J. Phys. A: Math. Theor. \textbf{47} (2014) 435201.
\bibitem{LHLO} H. J. Lee, K. Hong, H. Lee and S. Oh,
    {\em Mosaic number of knots},
    J. Knot Theory Ramifications \textbf{23} (2014) 1450069.
\bibitem{LK1} S. Lomonaco and L. Kauffman,
    {\em Quantum knots},
    Quantum Information and Computation II, Proc. SPIE \textbf{5436} (2004) 268--284.
\bibitem{LK2} S. Lomonaco and L. Kauffman,
    {\em Quantum knots and mosaics},
    Quantum Inf. Process. \textbf{7} (2008) 85--115.
\bibitem{OhD1} S. Oh,
    {\em State matrix recursion algorithm and monomer--dimer problem},
    (preprint).
\bibitem{OHLL} S. Oh, K. Hong, H. Lee and H. J. Lee,
    {\em Quantum knots and the number of knot mosaics},
    Quantum Inf. Process. \textbf{14} (2015) 801--811.
\end{thebibliography}
\end{document}